\documentclass{article}

\usepackage{amsmath,amsthm,amssymb}
\usepackage[all]{xy}
\usepackage{cmll}

\newtheorem{theorem}{Theorem}
\newtheorem{lemma}{Lemma}

\newtheorem{proposition}{Proposition}
\newtheorem{corollary}{Corollary}

\renewcommand{\root}{\ensuremath{\bot}}
\newcommand{\n}{\tilde{n}}
\renewcommand{\P}{\mathcal{P}}
\newcommand{\ID}{\ensuremath{\mathrm{ID}}}
\newcommand{\id}{\mathrm{id}}
\newcommand{\rk}{\mathrm{rk}}
\newcommand{\true}{\mathrm{True}}
\newcommand{\false}{\mathrm{False}}
\newcommand{\N}{\mathbb{N}}
\newcommand{\Z}{\mathbb{Z}}
\newcommand{\m}{\overline{n}}
\newcommand{\g}{\tilde{g}}
\title{The rank of the endomorphism monoid of uniformly nested partition.}
\author{Ivan Yudin\thanks{The work is supported by the FCT Grant SFRH/BPD/31788/2006. The
financial support by CMUC and FCT gratefully acknowledged.}
}

\begin{document}
\maketitle
\section*{Introduction}
Let $S$ be a semigroup. We say that $X$ \emph{generates} $S$ if every element of
$S$ can be written as a product of elements in $X$. The least possible
cardinality $\rk\left( S \right)$  of a generating set of $S$ is called the
\emph{rank} of $S$. 

There is a variety of works where the ranks of transformation monoids of
different structures are explicitly computed. In particular, in~\cite{j} the
rank of endomorphism monoid of uniform partition was shown to be $ 4$. 
In this paper we consider the endomorphism monoid of uniformly nested partition
and show that its rank is $2k$, where $k$ is the depth of nesting. 
\section{Nested partitions}
First we fix definition related to the notion of tree. As usually \emph{a
(rooted) tree} $T$ is a
simply connected graph with a fixed vertex $\root$ called root. We will indicate by
$t\in T$ that $t$ is a vertex of $T$. For $t_1$, $t_2\in T$ we define
$d(t_1,t_2)$ to be the number of edges at the unique path from $t_1$ to
$t_2$. Note that $d$ gives a distance function on the set of vertices of
$T$. We define a level of $t\in T$ as a distance from $t$ to $\rm root$.  
Denote by $T_k$ the set of all vertices in $T$ of level $k$. 
A vertex $s\in T_{k+1}$ is called \emph{a child} of $t\in T_k$ if there is an
edge between $s$ and $t$. 

A \emph{nested partition} of a set $X$ is a collection of subsets $\left\{\,
P_t \,\middle|\, t\in T \right\}$ of $X$ parametrized by the vertices of a tree
$T$ such that
\begin{list}{\alph{enumi})}{\usecounter{enumi}}
	\item $P_{\root} = X$;
	\item for any non-leaf $t\in T$:
		$$
		P_t = \coprod_{s \mbox{ \scriptsize child of } t} P_s.
		$$
\end{list}
We say that a map $f\colon X \to Y$ \emph{respects} nested partitions $\left\{\,
Q_s \subset X\,\middle|\,  s\in S \right\}$ and $\left\{\, P_t\subset Y
\,\middle|\,  t\in T
\right\}$ if for every $s\in S$ exists (necessarily unique) $t\in T$ of the same
level as $s$ such that $f\left( Q_s \right) \subset T_t$. 

For every nested partition $\left\{\, P_t \,\middle|\, t\in T \right\}$ of
$X$ we define sets $X_k$ by 
$$
X_k = \left(\coprod_{\begin{smallmatrix}t\in T_k\\ t \mbox{
	non-leaf}\end{smallmatrix}} \left\{ \mbox{childs of $t$}
\right\}\right) \amalg \left(\coprod_{\begin{smallmatrix}t\in T_k\\ t
	\mbox{ leaf}\end{smallmatrix}} P_t\right)
$$
and maps $\rho_k\colon X_{k+1}\to X_k$ by 
\begin{align*}
	\rho_k(s)& = t, && \mbox{if $s\in T_{k+2}$ is a child of $t\in
	T_{k+1}$}\\
	\rho_k(x) & = t, && \mbox{if $t\in T_{k+1}$ is a leaf and $x\in P_t$.}
\end{align*}
Note that if $T$ is a tree of depth $k$, then $X_l= \varnothing$ for $l\ge k+1$. 

If $f\colon X\to Y$ respects partitions $\left\{\, Q_s\subset X \,\middle|\,
s\in
S
\right\}$, $\left\{\, P_t\subset Y \,\middle|\, t\in T \right\}$ define
$f_k\colon X_k\to Y_k$
by 
\begin{align*}
	f_k(s)& = t & \mbox{if $s\in T_{k+1}$ and $f\left( Q_s \right)\subset
	P_t$}\\
	f_k(x)& = f(y) & \mbox{if  $s\in S_k$ is a leaf and $x\in P_s$}.
\end{align*}
\begin{proposition}
	The diagrams
	$$
	\xymatrix{X_{k+1} \ar[r]^{\rho_k}\ar[d]_{f_{k+1}} & X_k
	\ar[d]^{f_k}\\ Y_{k+1} \ar[r]^{\rho_k}& Y_k}
	$$
	are commutative. 
\end{proposition}
\begin{proof}
	Suppose $s'\in T_{k+2}$ is a child of $s\in T_{k+1}$. Then
	$f_{k+1}(s')=t'$ for $t'$ such that $f\left( Q_s \right)\subset
	P_{t'}$. Now $Q_{s'}\subset Q_s$ and $P_{t'}\subset P_t$, where
	$t$ is the parent of $t'$. Since $f\left( Q_s \right)\subset
	P_{\tilde{t}}$ for a unique $\tilde{t}\in P_{k+1}$, the subsets
	$P_{r}$, $r\in P_{k+1}$ of $X$ are disjoint, and $f\left( Q_s
	\right)\cap P_t \supset P_{t'}\not= \emptyset$ we get that $f\left( Q_s
	\right)\subset P_t$
	and $f_k(s)=t$.

	Now suppose $x\in Q_s$, where $s\in S_k$ is a leaf. Then
	$f_{k+1}(x)= f(x) \in P_t$, where $t= f_k(s)$. Therefore
	$f_k\rho_k(x)= f_k(s) = t = \rho_k\left( f(x) \right) =
	\rho_kf_{k+1}(x)$. 
\end{proof}
Suppose we have nested partition $\left\{\, Q_s\subset X \,\middle|\, s\in S
\right\}$, $\left\{\, P_t\subset Y \,\middle|\, t\in T \right\}$ and maps
$f_k\colon X_k\to Y_k$ such that the diagrams
$$
	\xymatrix{X_{k+1} \ar[r]^{\rho_k}\ar[d]_{f_{k+1}} & X_k
	\ar[d]^{f_k}\\ Y_{k+1} \ar[r]^{\rho_k}& Y_k}
$$
are commutative. Define $f\colon X\to Y$ as follows. For every $x\in X$ there is
a unique leaf $s\in S$ such that $x\in Q_s$. Suppose $s\in S_k$. Set
$f(x) := f_k(x)$. 
\begin{proposition}
	The map  $f\colon X\to Y$ defined above respects partitions 
$\left\{\, Q_s\subset X \,\middle|\, s\in S
\right\}$ and $\left\{\, P_t\subset Y \,\middle|\, t\in T \right\}$.
\end{proposition}
\begin{proof}
	Let $x\in Q_s\subset X$. There is a unique leaf $s'$ of $S$ such that
	$x\in Q_{s'}$. It is clear that $s'$ is a descendant of $s$. Suppose
	$s'\in S_l$ and $s\in S_k$, $l\ge k$. Then $x\in X_l$, $s'\in
	X_{l-1}$ and $s\in X_{k-1}$. We consider the commutative diagram
	$$
	\xymatrix@C2cm{
	X_l \ar[r]^{\rho_{l-1}} \ar[d]_{f_l} &
	X_{l-1}\ar[r]^{\rho_{k-1}\circ\dots\circ \rho_{l-2}}
	\ar[d]_{f_{l-1}} & X_{k-1} \ar[d]^{f_{k-1}} \\
	Y_l \ar[r]^{\rho_{l-1}} & Y_{l-1}\ar[r]^{\rho_{k-1}\circ\dots\circ
	\rho_{l-2}} & Y_{k-1}.
	}
	$$
	From this diagram it follows that $f(x) = f_l(x) \in
	P_{f_{k-1}(s)}$. Since $x$ was an arbitrary element of $Q_s$ we get that
	$f\left( Q_s \right)\subset P_{f_{k-1}}(s)$.
\end{proof}
We say that a nested partition $\left\{\, P_t \,\middle|\, t\in T \right\}$
of $X$ is \emph{uniformly nested} if all leaves of $T$ have depth $k$ for some
$k\in {\mathbb N}$ and there are natural numbers $n_1$, $n_2$, \dots, $n_k$ such
that $\left|P_t\right|= n_j$ for all $t\in T_j$. 
We will call $\left( n_1,\dots,n_k \right)$ a \emph{type} of the uniformly
nested partition. We define the \emph{standard} uniformly nested partition
$I(\tilde{n})$ of
type $\n = (n_1,\dots,n_k)$ as follows
$$
I\left( \n \right)_j := \left[ 1..n_1 \right]\times \dots \times
\left[ 1..n_j \right]
$$
for $0\le j\le k$ and
$$
\rho_j \colon I\left( \n \right)_{j+1}\to I(\n)_j
$$
to be the projection on the first $j$ coordinates. 
It is clear that every uniformly nested partition of type $\n$ is isomorphic to
$I(\n)$. Since endomorphism monoids  of isomorphic objects are isomorphic we
will concentrate on the endomorphism monoid of $I\left( \n \right)$. We will
denote it by $\P\left( \n \right)$. For every $v\in I\left( \n
\right)_{j-1}$ and $f\in \P\left( \n \right)$ define $f\left[ v \right]\colon
[1..n_{j}] \to \left[ 1..n_j \right]$ by the requirement
$$
f_j\left( v,i \right) = \left( f_{j-1}\left( v \right), f[v]\left( i
\right) \right). 
$$
\begin{proposition}
	\label{prop:composition}
	For every $f$, $g\in \P\left( \n \right)$, $v\in I\left( \n
	\right)_{j-1}$ we have $\left( fg \right)[v] = f\left[ g_{j-1}\left(
	v \right) \right]\circ g[v]$. 
\end{proposition}
\begin{proof}
	We have
	\begin{align*}
		\left( fg \right)_j\left( v,i \right) & = f_j\left(g_j\left(
		v,i \right) 
		\right) \\
		&= f_j\left( g_{j-1}\left( v \right), g[v]\left( i
		\right) \right) \\
		&= \left( f_{j-1}\left( g_{j-1}\left( v \right) \right), f\left[
		g_{j-1}\left( v \right) \right]\left( g[v]\left( i
		\right) \right)\right) \\
		&= \left( \left( fg \right)_{j-1}\left( v \right), \left(
		f\left[ g_{j-1}\left( v \right) \right]g\left[ v \right]
		\right)\left( i \right) \right).
	\end{align*}
\end{proof}
We can recover $f\in \P\left( \n \right)$ from the collection
$$
\left\{\, f[v]\colon [1..n_j]\to \left[ 1..n_j \right] \,\middle|\,v\in I\left(
\n \right)_{j-1}\right\}. 
$$
In fact we can define
$f_1 := f[*]$, where $*$ is the unique element of $I\left( \n \right)_0$. Then
we proceed by induction and define $f_j\left( v,i \right) := \left(
f_{j-1}\left( v \right), f[v]\left( i \right)
\right)$. 

For every $g\colon \left[ 1..n_j \right]\to \left[ 1..n_j \right]$ and
$v\in I\left( \n \right)_{j-1}$ we define $\left[ f,v \right]\in \P\left( \n
\right)$ by
$$
[g,v][w] := 
\begin{cases}
	f, & v= w\\
	\id, & \mbox{otherwise.}
\end{cases}
$$
Note that for $v_1$, $v_2\in I\left( \n \right)_{j-1}$ the elements $\left[
g,v_1
\right]$ and $\left[ g,v_2 \right]$ commute. We define $t_j\left( f
\right)$ to be the product of elements $\left[ f\left[ v \right],v \right]$
where $v$ ranges over $I\left( \n \right)_{j-1}$. As all this elements pairwise
commute the order in such a product does not play any role. We have
$$
t_j\left( f \right)\left[ v \right] = 
\begin{cases}
	f[v], & v\in I\left( \n \right)_{j-1}\\
	\id, & \mbox{otherwise.}
\end{cases}
$$
\begin{proposition}
	\label{lemma2.3a}
	\label{prop:identity}
	Let $v\in I\left( \n \right)_{j-1}$ and $g\colon \left[ 1..n_j
	\right]\to \left[ 1..n_j \right]$. Then for every $0\le s \le j-1$ we
	have $\left[ g,v \right]_s = \id$. 
\end{proposition}
\begin{proof}
	Let $\ID$ be the identity map of $I\left( \n \right)$. Then $\ID_s = \id$
	for every $0\le s \le k$. Therefore for every $w\in I\left( \n \right)_s$ we
	have $\ID\left[ w \right] = \id\colon [1..n_j]\to \left[ 1..n_j
	\right]$. Since the component  $f_s$ of $f\in \P\left( \n \right)$ can
	be recovered from the maps $f[w]$, $w\in I\left( \n \right)_r$,
	$r\le s$, and $\left[ g,v \right][w] = \id = \ID\left[ w \right]$ for
	all $w\in I\left( \n \right)_r$, $r\le j-1$, we see that $\left[ g,v
	\right]_s = \ID_s = \id$ for all $s\le j-1$. 
\end{proof}
\begin{corollary}
	\label{cor:identity}
	For every $f\in \P\left( \n \right)$ and $1\le s<j\le k$ we have
	$t_j\left( f \right)_s  = \id$. 
\end{corollary}
\begin{proposition}
	\label{prop:product}Let $g_1$, $g_2$ be endomorphisms of $\left[ 1..n_j
	\right]$ and $v\in I)\left( \n \right)_{j-1}$. Then $\left[ g_1,v
	\right]\left[ g_2,v \right] = \left[ g_1g_2,v\right],$. 
\end{proposition}
\begin{proof}
	Follows from the straightforward computation. 
\end{proof}
Now we will prove
\begin{proposition}
	\label{prop:decomposition}
	For every $f\in \P\left( \n \right)$ we have
	$$
	f = t_k\left( f \right)\circ \dots t_1\left( f \right). 
	$$
\end{proposition}
\begin{proof}
	We write $t_j$ instead of $t_j\left( f \right)$. 
	We have
	\begin{align*}
		\left( t_k\circ\dots t_1 \right)\left[ v \right] &= 
		t_k\left[ \left( t_{k-1\circ \dots \circ t_1}\left( v
		\right) \right) \right] \circ \left( t_{k-1}\circ\dots\circ t_1
		\right)[v]\\
		& = \left( t_{k-1}\circ \dots \circ t_1 \right) [v] \\
		& = \dots\\
		&= \left( t_j\circ \dots t_1 \right)[v]\\
		&= \left( t_j\circ \dots \circ t_2 \right)\left[ \left( t_1
		\right)_{j-1}\left( v \right)
		\right]\circ t_1\left[ v \right] \\
		&= \left( t_j\circ \dots \circ t_2 \right)[v]\\
		&= \dots \\
		&= t_j[v] = f[v]. 
	\end{align*}
\end{proof}
\section{Relative rank}
Let $S$ be a semigroup and $P\colon S \to \left\{ \true, \false \right\}$ a
predicate on $S$. We say that $P$ is \emph{primitive} if 
$$
\forall a,b\in S: P\left( ab \right) \Leftrightarrow P(a) \with P(b). 
$$
{\bf Examples} 1) Let $R$ be a commutative ring and $p$ a primitive ideal in
$R$. Then the predicate $P\left( x \right) := (x \notin p)$ is a primitive
predicate on the multiplicative semigroup of $R$. This example explains our
terminology. 

\noindent 2) Let $\mathcal{C}$ be a category and $X$ an object of
$\mathcal{C}$. Then the predicate
$$
\mbox{$f$ is an isomorphism}
$$
is a primitive predicate on $\mathcal{C}\left( X,X \right)$. 

\noindent 3) If $P_1$ and $P_2$ are primitive predicates then $P_1\&
P_2$ is primitive. 

	Suppose $P$ is a primitive predicate on $S$. We denote by $S_P$ the
	subset of $S$ of the elmements for which $P$ is true. Then $S_P$ is
	subsemigroup of $S$.  
\begin{proposition}
	\label{prop:relative}
	Let $P$ be a primitive predicate on $S$. Then
	$$
	\rk\left( S \right) = \rk\left( S_P \right) + \rk\left( S:S_P \right).
	$$
\end{proposition}
\begin{proof}
	It is obvious that $\rk\left( S \right)\le \rk\left( S_P \right) +
	\rk\left( S:S_P \right)$. Now let $X$ be a generating set of $X$ such
	that $\left| X\right|= \rk\left( S \right)$. Denote by $X_P$ the subset
	 $\left\{\, x\in X \,\middle|\,  P\left( x \right) \right\}$ of
	 $S_P$. Then $X_P$ generates $S_P$. In fact, let $s\in S_P$. Then
	 $s = x_1\dots x_m$ for some $x_i\in X$. Now
	 $$
	 P\left( s \right) \Leftrightarrow P\left( x_1\dots x_m \right)
	 \Leftrightarrow P\left( x_1 \right) \with \dots \with P\left( x_m
	 \right). 
	 $$
	 Thus $x_i\in X_P$ for every $1\le i\le m$. This shows that $s$ is an
	 element of subsemigroup generated by $X$. 

	 It is clear that $S_P\cup \left( X\setminus X_P \right)$ generates
	 $S$, since already its subset $X = X_P\cup \left( X\setminus X_P
	 \right)$ generates $S$. Therefore
	 $$
	 \rk\left( S \right) = \left|X\right| = \left|X\setminus X_P\right| +
	 \left|X_P\right| \ge \rk\left( S:S_P \right) + \rk\left( S_P
	 \right). 
	 $$
\end{proof}
For every $1\le j\le k$ we define the predicate $P_j$
on $\P\left( \n \right)$ by 
$$
P_j\left( f \right) = 
\begin{cases}
	\true, & \mbox{$f_j$ is invertible}\\
	\false, & \mbox{otherwise.} 
\end{cases}
$$
These predicates are primitive. We shall denote $\P\left( \n \right)_{P_j}$ by
$\P_j\left( \n \right)$. 
\begin{proposition}
	\label{prop:inclusion}
	We have $P_j\left( f \right)\Rightarrow P_{j-1}\left( f \right)$ and
	therefore $\P_j\left( \n \right)\subset \P_{j-1}\left( \n \right)$. 
\end{proposition}
\begin{proof}
	Since $f_{j-1}$ is an endomorphism of finite set it is non-invertible if
	and only if there are two different elements $v_1$, $v_2$ of $I\left( \n \right)_{j-1}$ such that
	$f_{j-1}\left( v_1 \right) =f_{j-1}\left( v_2 \right)= v $. We consider
	restriction $\overline{f}_j$
	of $f_j$
	on $\left\{ v_1,v_2 \right\}\times \left[ 1..n_j \right]$. The image of
	$\overline{f}_j$ is a subset of $\left\{ v \right\}\times \left[ 1..n_j
	\right]$. Thus $\overline{f}_j$ is not injective and therefore
	$P_j\left( f \right) = \false$. 
\end{proof}
\begin{proposition}
	\label{prop:invertible}
	Suppose $v\in I\left( \n \right)_{j-1}$ and $g\colon [1..n_j]\to
	[1..n_j]$. 
	Then $\left[ g,v \right]_j$ is invertible if and only if $g$ is
	invertible. Moreover, if $g$ is invertible then $\left[ g,v \right]\in
	\P_k\left( \n \right)$. 
	
	Then $\left[ g,v \right]\in \P_k\left( \n \right)$ if and
	only if $g$ is invertible. 
\end{proposition}
\begin{proof}
	For $s\le j-1$ we know by Proposition~\ref{lemma2.3a} that $\left[ g,v
	\right]_s=\id$. Now
	\begin{equation}	
		\label{eq:comp}
		\left[ g,v \right]_j\left( w,i \right)  = \left(
		[g,v]_{j-1}\left( w
		\right), \left[ g,v \right]\left[ w \right]\left( i
		\right)\right) 
		\begin{cases}
			\left( w,i \right), & w\not=v\\
			\left( w,g\left( i \right) \right), & w=v. 
			\end{cases}
			\end{equation}
			This shows that $\left[ g,v \right]_j$ is invertible if
			and only if $g$ is invertible. 

	Now we suppose that $g$ is invertible. Then $\left[ g,v \right]_j$ is
	invertible. Assume we showed that $\left[ g,v \right]_r$ is invertible
	for all $j\le r\le s$. Then
	$$
	\left[ g,v \right]_{s+1}\left( w,i \right) = \left( \left[ g,v
	\right]_s\left( w \right), i \right). 
	$$
	As $\left[ g,v \right]_s$ is invertible by assumption, it follows that
	$\left[ g,v \right]_{s+1}$ is invertible as well. 
\end{proof}
\begin{theorem}
	\label{thm:step}
For any $1\le j\le k$ we have
$$
\rk\left( \P_{j-1}\left( \n \right):\P_j\left( \n \right) \right) = 1. 
$$
\end{theorem}
\begin{proof}
	Since the inclusion of $\P_j\left( \n \right)$ in $\P_{j-1}\left( \n
	\right)$ is proper, the rank in question is at least $1$. Define
	$\tau\colon[1..n_j]\to \left[ 1..n_j \right]$ by 
	$$
	\tau\left( i \right) = 
	\begin{cases}
		2, & i= 1\\
		i ,& \mbox{otherwise.}
	\end{cases}
	$$
	Denote $\left( 1,\dots,1 \right)\in I\left( \n \right)_{j-1}$ by
	$u$. We claim that $\left[ \tau,u \right]$ generates $\P_{j-1}\left(
	\n \right)$ over $\P_j\left( \n \right)$. 

	Let $f\in \P_j\left( \n \right)$. By
	Proposition~\ref{prop:decomposition} we have
	$$
	f= t_k\left( f \right)\circ \dots \circ t_1\left( f \right),
	$$
	where 
	$$
	t_s\left( f \right) = \prod_{v\in I\left( \n \right)_{s-1}}\left[
	f\left[ v \right],v
	\right]. 
	$$
	Suppose $s\ge j+1$. Then by Corollary~\ref{cor:identity} $t_s\left( f
	\right)=\id$ and therefore $t_s\left( f \right)\in \P_j\left( \n
	\right)$. 

	Suppose $s\le j-1$ and $v\in I\left( \n \right)_{s-1}$. Since $f\in
	\P_{j-1}\left( \n \right)$ if follows that $\left[ f\left[ v
	\right],v \right]\in \P_{j-1}\left( \n \right)$. By
	Proposition~\ref{prop:inclusion} $\left[ f\left[ v \right], v \right]\in
	\P_s\left( \n \right)$ and therefore $\left[ f\left[ v \right],v
	\right]_s$ is invertible. By Proposition~\ref{prop:invertible}
	$f[v]$ is an automorphism of $\left[ 1..n_s \right]$. By the same
	proposition $\left[ f\left[ v \right],v \right]\in \P_k\left( \n
	\right)\subset \P_j\left( \n \right)$. Thus $t_s\left( f \right)\in
	\P_j\left( \n \right)$ for $s\le j-1$. 

	Let $v\in I\left( \n \right)_{j-1}$. Then $f\left[ v \right]\colon
	\left[ 1..n_j \right]\to \left[ 1..n_j \right]$ can be written as a
	product 
	$$
	g_1\tau g_2 \tau \dots \tau g_l
	$$
	for some $l\in \N$ and automorphisms $g_i$
	of $\left[ 1..n_j \right]$. By Proposition~\ref{prop:product} we have
	$$
	\left[ f\left[ v \right],v \right] = \left[ g_1,v \right]\left[ \tau,v
	\right]\dots \left[ \tau,v \right]\left[ g_l,v \right].
	$$
	By Proposition~\ref{prop:invertible} $[g_i,v]\in \P_j\left( \n
	\right)$ for $1\le i\le l$. Thus it is enough to show that $\left[
	\tau,v
	\right]$ belongs to a semigroup of  $\P_{j-1}\left( \n \right)$
	generated by $\P_j\left( \n \right)$ and $\left[ \tau,u \right]$.
	Suppose $v= \left( v_1,\dots,v_j \right)$, $1\le v_i\le n_i$. 
	Define
	$h\in \P\left( \n \right)$ by 
	$$
	h_s\left( w \right) := 
	\begin{cases}
		\left( v_1, \dots, v_s \right) & w = \left( 1,\dots,1 \right)\\
		\left( 1,\dots,1 \right) & w = \left( v_1,\dots,v_s \right)\\
		w & \mbox{otherwise}
	\end{cases}
	\mbox{if $s\le j-1$}
	$$
	$$
	h_s\left( w \right):= 
	\begin{cases}
		\left( v_1,\dots,v_{j-1},w_j,\dots,w_s \right) & \left(
		w_1,\dots,w_{j-1}
		\right) = u\\
		\left( 1,\dots,1,w_j,\dots,w_s \right) & \left(
		w_1,\dots,w_{j-1} \right) = v\\
		w & \mbox{otherwise}
	\end{cases} \mbox{if $s\ge j$.} 
	$$
	Then $h^2 = \ID$ and thus $h\in \P_{j}\left(\n  \right)$. Moreover
	$\left[ \tau,v \right] = h\left[ \tau,u \right]h$. This finishes the
	proof. 
\end{proof}
\begin{corollary}
	\label{cor:relative} We have 
	$$
	\rk\left( \P\left( \n \right) \right) = k + \rk\left( \P\left( \n
	\right):\P_k\left( \n \right) \right). 
	$$
\end{corollary}
\begin{proof}
	Apply Theorem~\ref{thm:step} and Proposition~\ref{prop:relative}. 
\end{proof}
\section{Generators for wreath product}

We will multiply permutations from left to right, thus
$$
\left( 1,2 \right)\left( 2,3 \right) = \left( 1,3,2 \right).
$$
Correspondingly, if $\pi\in S_m$ and $i\in \left\{ 1,\dots,m \right\}$, then the
result of application of $\pi$ to $i$ will be denoted by $i\pi$. 

Let $G$ be a group. We define a left action of
$S_m$ on $G^m$ by 
$$
\pi \left( g_1,g_2,\dots,g_m \right)= \left(
g_{1\pi},g_{2\pi},\dots,g_{m\pi}  \right).
$$
Then the multiplication in the wreath product $G\wr S_m = G^m \rtimes S_m$ is given
by 
$$
\left( h_1,\dots,h_m \right) \pi \left( g_1,g_2,\dots,g_m \right)\sigma = \left(
h_1g_{1\pi},h_2g_{2\pi},\dots,h_mg_{m\pi} \right)\pi\sigma.
$$ We will consider $G^m$ and $S_m$ as subgroups of $G\wr S_m$. 
Denote by $[-,i]$ the embedding of $G$ into the $j$-th component of
$G^m$.Then the set 
$$
\left\{\, [g,i] \,\middle|\, g\in G,\ i\in \left\{ 1,\dots,m \right\} \right\}
\cup S_m
$$
generates $G\wr S_m$.
In the following we will use that for a different $i$ and $j$ the elements
$[g,i]$ and $[h,j]$ of $G\wr S_m$ commute, and that for $\pi\in S_m$
$$
\pi[g,i] = [g,i\pi^{-1}]\pi.
$$
\begin{proposition}
	\label{prop:gen}
	Suppose $\left\{ g_1,\dots,g_k \right\}$ and $\left\{ \pi_1,\dots,\pi_l
	\right\}$ are generating sets of $G$ and $S_m$ respectively. Then for
	any multi-index $\left( i_1,\dots,i_k \right)$, $1\le i_t\le m$, the set
	$$
	X=\left\{\, [g_t,i_t] \,\middle|\, 1\le t\le k \right\} \cup \left\{
	\pi_1,\dots,\pi_l
	\right\}
	$$
	generates $G^m\wr S_m$. 
\end{proposition}
\begin{proof}
	Denote by $H$ the subgroup of $G$ generated by $X$. Then $S_m\subset H$
	as $\left\{ \pi_1,\dots,\pi_l \right\}\subset X$ and $\left\{
	\pi_1,\dots,\pi_l
	\right\}$ generates $S_m$. Now for every $1\le t\le k$ and $1\le j\le m$
	we have
	$$
	(i_t,j)[g_t,i_t](i_t,j) = [g_t,j]\in H.
	$$
	Since the $\left\{\, [g_t,j] \,\middle|\, 1\le t\le k \right\}$
	generates subgroup $[G,j]$ of $G\wr S_m$ we get that
	$$
	\left\{\, [g,i] \,\middle|\, g\in G,\ i\in \left\{ 1,\dots,m \right\} \right\}
\cup S_m\subset H
	$$
	and therefore $H= G\wr S_m$.
\end{proof}

Now we prove two lemmas that show how the elements of the form $[g,i]$ and
$\pi$ can be recovered from the elements of the form $[g,i]\pi$.
\begin{lemma}
	\label{l:com}
	\label{coprime}
	Suppose $g\in G$ and $\pi\in S_m$ have coprime orders. Let $i\in \left\{
	1,\dots,m \right\}$ be such that $i\pi=i$. Then $[g,i]$ and $\pi$ are
	elements of the cyclic subgroup generated by $[g,i]\pi$. 
\end{lemma}
\begin{proof}
	Note that $[g,i]$ and $\pi$ commute since $\pi [g,i] =
	[g,i\pi]\pi=[g,i]\pi$.
	Let $k$ and $l$ be the orders of $g$ and $\pi$ respectively. Since
	$k$ and $l$ are coprime
	there are $p$ and $q$ such that $pk + ql = 1$. Now $\left( [g,i]\pi
	\right)^{pk}= \left[ g^{pk},i \right]\pi^{1-ql} = \pi$ and $\left(
	[g,i]\pi
	\right)^{ql} = \left[ g^{1-pk},i \right]\pi^{ql} = [g,i]$.
\end{proof}
\begin{lemma}
	\label{strannaya}
	\label{l:gen}
	Let $g$ be an element of $G$ of odd order and $\sigma\in G$ of order
	$2$. Denote by $H$ the subgroup of $G\wr S_m$ generated by 
	\begin{align*}
	a = &	[\sigma,2]\left( 1,\dots,m \right)\\
		b = & [ g,3]\left( 1,2 \right).
	\end{align*}
	Then $[g,3]$, $[\sigma,1]$, $\left( 1,\dots,m \right)$ and $\left( 1,2
	\right)$  are elements of $H$.
	\end{lemma}
\begin{proof}
	Since the order of $g$ is odd, the order of $\left( 1,2 \right)$ is
	$2$,  and $[g,3]\left( 1,2
	\right)\in H$  it follows from Lemma~\ref{coprime} that 
	$\left[ g,3 \right]$ and $\left( 1,2 \right)$ are elements of $H$.

Now we  have that
\begin{align*}
a^m &=\left( \left[ \sigma,2 \right]\left( 1,\dots,m \right)\right)^m =
[\sigma,2][\sigma,3]\dots [\sigma,m][\sigma,1]\left( 1,\dots,m\right)^m\\ & =
\left( \sigma,\sigma,\dots,\sigma \right)\in H.
\end{align*}
Consider the product
\begin{align*}
ab& = [\sigma,2]\left( 1,\dots,m \right) 
[g,3]\left( 1,2 \right) = [\sigma,2][g,3\left( 1,\dots,m \right)^{-1}]\left(
2,\dots,m \right)\\ & = [\sigma g, 2] \left( 2,\dots,m \right).
\end{align*}
	Therefore
	\begin{align*}
		\left( ab \right)^{m-1} & = \left([\sigma g, 2] \left( 2,\dots,m
		\right)\right)^{m-1} = [\sigma g,2] [\sigma g,3] \dots [\sigma
		g,m]\left(
	2,\dots,m
	\right)^{m-1}\\ & = \left( e,\sigma g,\dots,\sigma g \right) \in H
\end{align*}
		and, using $\sigma^2=e$, 
$$
a^m\left( ab \right)^{m-1} = \left( \sigma,g,\dots,g \right)\in H.
$$
Let $l$ be the order of $g$. Then since $l$ is odd and the order of $\sigma$ is
$2$ we have $\sigma^l=\sigma$. Thus
$$
\left( a^m\left( ab \right)^{m-1} \right) = \left( \sigma,e,\dots,e \right)=
[\sigma,1] \in
H.
$$
Note that $a^{-1}= [\sigma,3]\left( 1,m,\dots,2 \right)$. 
Therefore
\begin{align*}
	a^{-1}[\sigma,1]a^2 &= [\sigma,3]\left( 1,m,\dots,2 \right)
	[\sigma,1][\sigma,2]\left( 1,\dots,m \right) [\sigma,1]\left(1,\dots,m
	\right) \\ &
	= [\sigma,3] [\sigma,2][\sigma,3]\left( 1,m,\dots,2 \right) \left(
	1,2,\dots,m
	\right) [\sigma,2] \left( 1,\dots,m \right)
\\&= [\sigma,2][\sigma,2]\left( 1,2,\dots,m \right)	
	= \left( 1,\dots,m \right)
	\in H.
\end{align*}
\end{proof} 
\section{Iterated wreath product}
In this section we identify $\P_k\left( \n \right)$ with an iterated wreath
product and show that its rank is $k$. 

Let $\Z_2={e,\sigma}$ be the cyclic group of order $2$. Denote by
$\varepsilon_j$ the parity homomorphism  from $S_{n_j}$ to $\Z_2$. Define
$\varepsilon\colon \P_k\left( \n \right)\to \Z^k_2$ by 
$$
\varepsilon\left( f \right) = \left( \varepsilon_1\left( f_1 \right), \dots,
\varepsilon_k\left( f_k \right) \right).
$$
\begin{proposition}
	\label{prop:hom}
	The homomorphism $\varepsilon$ is surjective and therefore $\rk\left(
	\P_k\left( \n \right) \right)\ge k$. 
\end{proposition}
\begin{proof}
	Let $g_j=\left[ \left( 1,2 \right), u_j \right]$, where $u_j=\left(
	1,\dots,1
	\right)\in I\left( \n \right)_{j-1}$. Then by
	Proposition~\ref{prop:identity} for $s\le j-1$
	$$
	\varepsilon_s\left( (g_j)_s \right) = \varepsilon_s\left( \id \right)=
	e.
	$$
	Now by \eqref{eq:comp}
	$$
	\left( g_j \right)_j \left( w,i \right) = 
	\begin{cases}
		\left( w,i \right) & w\not= u_j\\
		\left( w,\left( 1,2 \right)\left( i \right) \right) & w= u_j.
	\end{cases}
	$$
	Thus $\left( g_j \right)_j$ swaps two elements $\left( u_j,1 
	\right)$  and $\left( u_j,2 \right)$ of $I\left( \n \right)_{j}$.
	Therefore $\varepsilon_j\left( \left( g_j \right)_j \right) = \sigma$.
	Since the elements 
	\begin{align*}
		\left( \sigma,*,\dots,* \right)\\
		\left( e,\sigma, *,\dots, * \right)\\
		\dots\\
		\left( e,\dots,e,\sigma \right)
	\end{align*}
	generate $\Z_2^k$ by Gauss elimination process, we see that
	$\varepsilon$ is surjective. 
\end{proof}

Let $\m = \left( n_2,\dots,n_k \right)$. We can identify $\P_k\left( \n
\right)$ with $\P_{k-1}\left( \m \right)\wr S_{n_1}$ as follows. Let $f\in
\P_k\left( \n \right)$ and $1\le i\le n_1$. Define $f\left( i \right)\in
\P_{k-1}\left( \m \right)$ from the equalities
$$
\left( f\left( i \right)_j\left( v \right), f_1\left( i \right) \right) =
f_{j+1}\left( v,i \right).
$$
Then 
$$
f\mapsto \left( f\left( 1 \right),\dots,f\left( n_1 \right) \right) f_1
$$
is an isomorphism from $\P_k\left( \n \right)$ to $\P_{k-1}\left( \m
\right)\wr S_{n_1}$. By iteration we get that 
$$
\P_k\left( \n \right)\cong S_{n_k} \wr S_{n_{k-1}} \wr \dots \wr S_{n_1}. 
$$
Let $\sigma\in S_{n_j}$ and $\left( v_1,\dots,v_{j_1} \right)\in I\left( \n
\right)_{j-1}$. Then upon this identification
$$
\left[ \sigma,u \right] = \left[ \left[ \left[ \sigma,v_{j-1} \right],\dots
\right], v_1 \right]. 
$$

For every $j$ we define an element $\tau_j$ of $S_{n_j}$ by
$$
\tau_j = 
\begin{cases}
	\left( 1,\dots,n_j \right) & \mbox{$n_j$ is odd}\\
	\left( 2,\dots,n_j \right) & \mbox{$n_j$ is even}.
\end{cases}
$$
Note that the order of $\tau_j$ is odd for every $j$ and that the
elements $(1,2)$, $\tau_j$ generate $S_{n_j}$. 

Define 
$$
g_j = \left[ \tau_{j+1},3 \right]\left( 1,2 \right) \in S_{n_k} \wr \dots \wr
S_{n_j},
$$
 and 
 $$
 \g_j = \left[ \left[ g_j,3 \right],\dots,3 \right]\in S_{n_k} \wr \dots \wr
S_{n_1},
 $$
 Define
$$
\g =\left[ \left[ \left[ \left( 1,2 \right),2 \right],\dots \right],2 \right]
 S_{n_k} \wr \dots \wr S_{n_1}.
$$
\begin{theorem}
	\label{t:rank}The set $X =\left\{ \g_1,\dots,\g_{k-1}, \g \right\}$
	generates $G=S_{n_k} \wr \dots \wr S_{n_1}$ and therefore $\rk\left(
	G \right)\le k$. 	
\end{theorem}
\begin{proof}
	Let $H$ be a subgroup of $G$ generated by $X$. 
	Denote by $u_j$ the element $\left( 3,\dots,3 \right)$ of $I\left( \n
	\right)_{j-1}$. 
	
	As $\g_j\in H$, it
	follows from Lemma~\ref{l:com} that for $1\le j\le k-1$
	$$
\left[ \left( 1,2 \right), u_j \right]\in H
	$$
	and for $2\le j\le k$ 
	$$
	\left[ \tau_j,u_j \right]\in H.
	$$
	Now we apply Lemma~\ref{l:gen}
	to $\g_1$ and $\g$. We get that $\left( 1,\dots,n_1 \right)$ is an
	element of $H$. By iteration of Proposition~\ref{prop:gen} we get that
	these elements generate the whole group $G$. 
\end{proof}
\begin{theorem}
	\label{main}
	We have $\rk\left( \P\left( \n \right) \right) =2k$. 
\end{theorem}
\begin{proof}
	From  Theorem~\ref{t:rank} and Proposition~\ref{prop:hom} it follows
	that $\rk\left( \P_k\left( \n \right) \right) = k$. Now apply
	Corollary~\ref{cor:relative}. 
\end{proof}

\bibliography{np}
\bibliographystyle{amsplain}

\end{document}